\RequirePackage{iftex}
\RequirePackage{fix-cm}
\documentclass[a4paper,11pt]{article}
\usepackage[top=25mm,bottom=26mm,left=25mm,right=25mm]{geometry}
\usepackage{amsmath,amssymb,amsthm,mathtools,mathrsfs}
\usepackage{tikz}
\input{ursfs.fd}
\DeclareFontShape{U}{rsfs}{m}{n}{<-6>rsfs5<6-8>rsfs7<8->rsfs10}{}
\usepackage{booktabs,array}
\usepackage[unicode,hidelinks]{hyperref}
\hypersetup{pdftitle={Equivariant Moment Maps and the Ricci Cocycle in Weighted Fujiki--Donaldson Theory},
pdfauthor={},pdfsubject={・2026.10.1}}
\allowdisplaybreaks[2]
\numberwithin{equation}{section}
\newtheorem{theorem}{Theorem}[section]
\newtheorem{proposition}[theorem]{Proposition}
\newtheorem{lemma}[theorem]{Lemma}

\theoremstyle{definition}

\theoremstyle{remark}
\newtheorem{remark}[theorem]{Remark}

\newtheorem{problem}[theorem]{Problem}

\newcommand{\R}{\mathbb R}
\def\C{\mathbb C}
\newcommand{\Z}{\mathbb Z}
\newcommand{\A}{\mathcal A}
\newcommand{\J}{\mathcal J}
\newcommand{\gh}{\mathfrak g_h}
\newcommand{\gex}{\mathfrak g_{\mathrm{ex},\nu}}
\newcommand{\Ham}{\operatorname{Ham}}

\newcommand{\Ric}{\operatorname{Ric}}
\newcommand{\tr}{\operatorname{tr}}
\newcommand{\Log}{\operatorname{Log}}
\newcommand{\diver}{\operatorname{div}}
\newcommand{\grad}{\operatorname{grad}}
\newcommand{\Lie}{\mathcal L}
\newcommand{\id}{\mathrm{id}}
\newcommand{\ii}{\mathrm i}
\newcommand{\DD}{\mathscr D}
\newcommand{\dn}{\,d\nu}

\newcommand{\norm}[1]{\left\lVert#1\right\rVert}
\newcommand{\pair}[2]{\left\langle#1,#2\right\rangle}
\newcommand{\HdR}{H_{\mathrm{dR}}}
\newcommand{\Hcont}{H^2_{\mathrm{cont}}}

\def\ome{\omega}

\newcommand{\nablau}{\widehat\nabla}
\DeclareMathOperator{\Imag}{Im}

\title{Equivariant Moment Maps and the Ricci Cocycle in Weighted Fujiki--Donaldson Theory}
\author{Ryushi Goto}
\author{%
  Ryushi Goto\\[0.5em]
  {\small Department of Mathematics, Graduate School of Science}\\
  {\small The University of Osaka}\\
  {\small Toyonaka, Osaka 560-0043, Japan}\\
  {\small\texttt{goto@math.sci.osaka-u.ac.jp}}
}
\date{}
\begin{document}
\maketitle
\begin{abstract}
Let $(M,\omega)$ be a closed connected symplectic manifold of
dimension $2n$.
We discuss weighted Fujiki--Donaldson theory with respect to
the measure $d\nu=e^h\omega^n/n!$ for an arbitrary smooth real-valued
function $h$ on $M$.
We prove that the group generated by the flows of Hamiltonians
that Poisson-commute with $h$ always admits an equivariant moment
map for its action on the full space of $\omega$-compatible
almost complex structures.
The moment map takes values in the continuous dual of the Lie
algebra.
Its density is expressed in terms of the Hermitian scalar
curvature, a local weight correction, and an $L^\infty$ function
$r_h^0$ constructed from the long-time average of the polar
determinant phase.
The function $r_h^0$ can be regarded as a Ruelle-type asymptotic
density incorporating phase comparison via a connection.
The key step in the proof is to represent the Ricci cocycle as
the coboundary of a continuous cochain constructed from $r_h^0$.

We also consider the action of the exact volume-preserving
diffeomorphism group on a connected component of the space of
almost complex structures inducing the given orientation
on a closed connected oriented manifold of real dimension $2n$.
In contrast to the weighted Hamiltonian setting, we prove that,
for $n>1$, this action admits an equivariant moment map if and
only if the real first Chern class vanishes.


\end{abstract}


\section{Introduction}
Fujiki~\cite{Fujiki} and Donaldson~\cite{Donaldson} showed that,
on a closed symplectic manifold, the action of the group of
Hamiltonian diffeomorphisms on the space of compatible almost
complex structures admits a moment map given by the Hermitian
scalar curvature.
This picture has been extended to generalized K\"ahler geometry:
a moment map has been constructed in the toric
case~\cite{Bou_2019} and for generalized K\"ahler structures of
symplectic type~\cite{Goto_2016}.
However,
in the general case, without an invariance assumption,  the existence of such a moment map
remains open; the difficulties associated with generalized
Hamiltonian vector fields are discussed in~\cite{Goto_2026}.
In the ordinary almost K\"ahler setting, the corresponding
vector field is the Hamiltonian vector field $X_h$ associated
with the weight $h$.
This motivates the study of moment maps in weighted
Fujiki--Donaldson theory for arbitrary weights, without imposing
the invariance condition $\Lie_{X_h}J=0$.

To state the problem, let $(M^{2n},\omega)$ be a closed connected
symplectic manifold, and fix an arbitrary $h\in C^\infty(M,\R)$.
We write
\begin{equation}\label{eq:measure}
 dV=\frac{\omega^n}{n!},\qquad d\nu=e^{h}\,dV.
\end{equation}
Our conventions for Hamiltonian vector fields and the Poisson
bracket are
\begin{equation}\label{eq:hamconventions}
 \iota_{X_f}\omega=-df,\qquad \{f,g\}=X_f(g),\qquad
 [X_f,X_g]=X_{\{f,g\}}.
\end{equation}
Consider the Lie algebra
\begin{equation}\label{eq:gh}
 \gh=\left\{f\in C^\infty(M,\R):
       \{h,f\}=0,\ \int_M f\,d\nu=0\right\},
\end{equation}
and let $G_h=\Ham(M,\omega,h)$ be the group generated by the flows
of time-dependent Hamiltonians in $\gh$.
Let $\A(M,\omega)$ denote the space of $\omega$-compatible almost
complex structures on $M$. It carries the natural symplectic form
\begin{equation}\label{eq:omegaA}
 \Omega_{h,J}(A,B)=\frac12\int_M\tr(JAB)\,d\nu .
\end{equation}
We let $G_h$ act on $\A(M,\omega)$ by pullback, $J\mapsto\phi^*J$,
and use the following convention for moment maps:
\begin{equation}\label{eq:momentconvention}
 D_J\pair{\mu(J)}{f}[A]=\Omega_{h,J}(A,\Lie_{X_f}J).
\end{equation}
The question we address is the following.

Does the action of $G_h$ on $(\A(M,\omega),\Omega_h)$ admit an
equivariant moment map?

In weighted theories, moment maps given by weighted scalar
curvatures have been obtained on spaces of structures invariant
under the relevant symmetry~\cite{La_2019,Ino_2022}.
We call $J\in\A(M,\omega)$ \emph{adapted} if $\Lie_{X_h}J=0$.
An equivariant moment map on the adapted locus was constructed
in~\cite{ASU_2026}, and whenever this locus is nonempty, the moment
map extends to the whole space $\A(M,\omega)$~\cite{Goto_2026}.

In this paper we answer this question affirmatively for
every smooth weight $h$. We assume neither that $J$ is adapted
nor that an adapted structure exists.

\begin{theorem}[Existence for arbitrary weights]\label{thm:weighted}
Let $h\in C^\infty(M,\R)$ be arbitrary.
Then the action of $G_h$ on $(\A(M,\omega),\Omega_h)$ admits an
equivariant moment map with values in the continuous dual $\gh'$.
Explicitly, one may take
\begin{equation}\label{eq:mainmoment}
 \pair{\mu_h(J)}{f}
 =\int_M f\left(S_J-\Delta_J h-|dh|_{g_J}^2-2r_h^0(J)\right)d\nu .
\end{equation}
Here $g_J=\omega(\,\cdot\,,J\,\cdot\,)$,
$\Delta_J=\diver_{g_J}\grad_{g_J}$, and
$S_J=2\Lambda_\omega\rho_J$ is the Hermitian scalar curvature.
The function $r_h^0(J)\in L^\infty(M,d\nu)$ is the long-time
average of the phase of the polar determinant along the flow of
$X_h$, measured relative to the canonical determinant connection
(see Section~\ref{sec:phase}).
For every $f\in\gh$, the pairing $\pair{\mu_h(J)}{f}$ depends
smoothly on $J$ and satisfies
\begin{equation}\label{eq:equivgroup}
 \pair{\mu_h(\phi^*J)}{f}
 =\pair{\mu_h(J)}{f\circ\phi^{-1}}
 \qquad(\phi\in G_h).
\end{equation}
\end{theorem}

Theorem~\ref{thm:weighted} requires neither $c_1(M)=0$ nor
$c_1(M)\in\R[\omega]$, and it does not assume the existence of an
adapted structure.
On the other hand, the moment map takes values in the continuous
dual $\gh'$; we do not claim that the density in~\eqref{eq:mainmoment}
is smooth.
Note also that for nonintegrable $J$ the Hermitian scalar
curvature $S_J$ differs in general from the Riemannian scalar
curvature of $g_J$.

The proof combines long-time phase averaging with two comparison
arguments. Let $\phi_t$ be the flow of $X_h$. The bounded defect
of the polar determinant phase of $D\phi_T$, measured relative to
the determinant connection, gives almost additivity in $T$.
Kingman's subadditive ergodic theorem~\cite{Kingman} then yields
the long-time density $r_h^0(J)$. This construction is related to
Ruelle-type asymptotic invariants of symplectic cocycles~\cite{CE},
but requires no global trivialization of the determinant line bundle.
The first comparison, around rectangles formed by commuting
Hamiltonian flows, shows that the Ricci cocycle
\[
 C_{\rho_J,h}(f,g)
 =\int_M\rho_J(X_f,X_g)\,d\nu,\qquad f,g\in\gh,
\]
is a continuous coboundary, removing the obstruction to
equivariance. The second comparison, along paths of compatible
almost complex structures, gives a finite-variation formula for
$\int_M f r_h^0(J)\,d\nu$. This establishes the moment-map
differential identity without interchanging differentiation in
$J$ with the long-time limit.

When $J$ is adapted, we have
$r_h^0(J)=\frac12\Delta_J h$, and the moment-map density reduces to
\[
 S_J-2\Delta_J h-|dh|_{g_J}^2.
\]
Adapted structures need not exist: we show that if $h$ is a
Morse function with a critical point of odd Morse index, then
no $\omega$-compatible almost complex structure is adapted
(see Section~\ref{sec:regularity}).
Theorem~\ref{thm:weighted} also applies when the adapted
locus is empty.


We next turn to the analogous problem for the group of exact
volume-preserving diffeomorphisms.
Here both the acting group and the space of structures are
different.
Garc\'ia-Prada, Salamon and Trautwein~\cite{GST} showed that in
this setting an equivariant moment map exists if the real first
Chern class vanishes.
We prove the converse in real dimension at least four, using the
computation of continuous Lie algebra cohomology due to Janssens,
Ryvkin and Vizman~\cite{JRV}.
Consequently, in contrast with the weighted Hamiltonian setting
above, the vanishing of the real first Chern class is both
necessary and sufficient for the existence of an equivariant
moment map in real dimension at least four.

Let $M^{2n}$ be a closed connected oriented manifold with a
positive volume form $\nu$, and set
\begin{equation}\label{eq:gex}
 \gex=\{v\in\Gamma(TM):\iota_v\nu\ \text{is exact}\}.
\end{equation}
Let $G_{\mathrm{ex},\nu}$ be the group generated by the flows of
time-dependent vector fields in $\gex$.
Let $\J_*$ be a connected component of the space $\J(M)$ of almost
complex structures inducing the given orientation.
It carries the symplectic form
\[
 \Omega_{\nu,J}(A,B)=\frac12\int_M\tr(JAB)\,\nu .
\]

\begin{theorem}[Existence criterion for the exact volume-preserving group]
\label{thm:exactvolume}
Suppose that $n>1$.
The action of $G_{\mathrm{ex},\nu}$ on $(\J_*,\Omega_\nu)$ admits
an equivariant moment map with values in the continuous dual
$\gex'$ if and only if
\[
 c_1(TM,J)=0\quad\text{in }H^2(M;\R)
\]
for some (equivalently, every) $J\in\J_*$.
In that case, choose for each $v\in\gex$ a form $\alpha_v$ with
$d\alpha_v=\iota_v\nu$. Then
\begin{equation}\label{eq:mainexact}
 \pair{\mu_{\mathrm{ex}}(J)}{v}
 =-2\int_M\Ric_{\nu,J}\wedge\alpha_v
\end{equation}
defines an equivariant moment map.
\end{theorem}

\subsection*{Open problems}
The moment map of Theorem~\ref{thm:weighted} is represented by an
$L^\infty$ density on $M$, although its pairing with every
$f\in\gh$ depends smoothly on $J$. It remains unclear whether
it always admits a smooth representing density. This is distinct
from the regularity of $r_h^0(J)$ itself, since representing
densities are only determined modulo the annihilator of $\gh$
(see Section~\ref{sec:regularity}).
\begin{problem}\label{prob:regularity}
 When does $r_h^0(J)$ agree with a smooth function
modulo the annihilator of $\gh$?
\end{problem}

Theorem~\ref{thm:weighted} resolves the weighted almost K\"ahler
counterpart of the generalized K\"ahler problem that motivated
this work. 
\begin{problem}\label{prob:gk}
In the generalized K\"ahler framework of~\cite{Goto_2026},
does the generalized Hamiltonian action admit an equivariant
moment map on the full space of compatible almost generalized
complex structures, without assuming that the adapted locus
is nonempty?
Can such a map be constructed using an analogue of the long-time
phase averaging developed here?
\end{problem}


\subsection*{Organization of the paper}
The paper is organized as follows.
In Section \ref{sec:ricci}, 
we fix our conventions for the Ricci
form and record the formulas used throughout the paper.
In Section \ref{sec:phase}, we construct the long-time phase.
The rectangle comparison in Section \ref{sec:cocycle} expresses
the Ricci cocycle as a continuous coboundary.
In Section \ref{sec:variation}, we establish a variation formula with respect
to $J$, and in Section \ref{sec:momentproof}, we complete the proof
of Theorem \ref{thm:weighted}.
In Section \ref{sec:regularity}, 
we discuss regularity and
the existence and nonexistence of adapted structures.
Finally, in Section \ref{sec:exactvolume}, we prove
Theorem \ref{thm:exactvolume} using 
the computation of the second continuous Lie algebra cohomology.


\section{Geometric Framework and the Ricci Form}\label{sec:ricci}

\subsection{The Continuous Dual and Conventions for Equivariance}
We equip $\gh$ with the Fr\'echet topology induced by $C^\infty(M)$.
Since it is the intersection of the kernel of $X_h$ with the
zero-mean subspace, it is a closed Fr\'echet subspace.
For $f\in\gh$, the vector field $X_f$ preserves $\omega$, $h$, and
$d\nu$, and $\int_M\{f,g\}\dn=0$ for $f,g\in\gh$.
Thus \eqref{eq:gh} defines a Lie algebra.
We denote by $\gh'$ its continuous dual with respect to this topology.
In this paper, a dual-valued map is called smooth if its pairing
with each fixed $f\in\gh$ is smooth.

The form \eqref{eq:omegaA} is obtained by integrating the invariant
closed two-form on the symmetric space of compatible complex structures
at each point against a fixed positive measure.
It is therefore closed, and its pointwise nondegeneracy, together
with localized variations, implies weak nondegeneracy.
Since $G_h$ preserves $\nu$, it also preserves this two-form.

Write $V_f(J)=\Lie_{X_f}J$.
With the convention \eqref{eq:momentconvention}, infinitesimal
equivariance takes the form
\begin{equation}\label{eq:infequiv}
 \Omega_{h,J}(V_f,V_g)+\pair{\mu_h(J)}{\{f,g\}}=0.
\end{equation}
This is equivalent to
$D\pair{\mu_h}g[V_f]=-\pair{\mu_h}{\{f,g\}}$.
Since pullback defines a right action, we specify group equivariance
by \eqref{eq:equivgroup}, avoiding potentially ambiguous
$\operatorname{Ad}^*$ notation.

\subsection{Notation and known results for almost K\"ahler geometry}
Let $J$ be an almost complex structure compatible with $\omega$.
The volume form is denoted by $dV=\frac{\ome^n}{n!}.$ 
Let $\nabla^g$ denote the Levi--Civita connection of
$g_J=\omega(\cdot,J\cdot)$, and define
\begin{equation}\label{the first canonical Hermitian conn}
 \widehat \nabla_XY
 =\nabla^g_XY-\frac12 J(\nabla^g_XJ)Y,
\end{equation}
which is the first canonical Hermitian connection;
see \cite{Gauduchon1997}.
Then $\widehat \nabla$ is a Hermitian connection preserving both $g_J$ and $J$.
Thus $\widehat \nabla$ defines a connection on $T_J^{1,0}$.
The associated Ricci form is
\[
 \Ric_{dV,J}
 =\frac12\tr\bigl(JR^{\widehat{\nabla}}\bigr);
\]
see \cite[Theorem~2.6(iii), equation~(2.18)]{GST}.
More generally, for a positive volume form $\nu$ and an almost
complex structure $J$ compatible with the orientation, let
$\Ric_{\nu,J}$ denote the Ricci form introduced by
Garc{\'i}a-Prada--Salamon--Trautwein
\cite[Definition~2.5]{GST}.
By \cite[Theorem~2.6(i), (iii)]{GST}, it is a closed real
$2$-form satisfying
\begin{equation}\label{eq:ricciclass}
  [\Ric_{\nu,J}]=2\pi c_1(TM,J),\qquad
  \Ric_{e^u\nu,J}=\Ric_{\nu,J}+\frac12d(du\circ J).
\end{equation}
Moreover, it is natural under diffeomorphisms.
Throughout this paper, we denote the Ricci form and the
Hermitian scalar curvature with respect to $dV$ by
\[
 \rho_J=\Ric_{dV,J},\qquad
 S_J=2\Lambda_\omega\rho_J,
\]
respectively, where $\Lambda_\omega$ denotes contraction
with the symplectic form $\omega$.

Fix a torsion-free symplectic connection $\nabla$.
For $A\in T_J\A(M,\omega)$, define the one-form
$\Lambda_{dV}(J,A)\in\Omega^1(M)$ by
\begin{equation}\label{eq:Lambda}
 \Lambda_{dV}(J,A)(u)
 :=
 \tr\left((\nabla A)u+\frac12 AJ\nabla_uJ\right),
 \qquad u\in\Gamma(TM),
\end{equation}
following \cite[Definition~2.5, equation~(2.9)]{GST}.
Here $(\nabla A)u$ denotes the endomorphism
\[
 v\longmapsto(\nabla_vA)u,
\]
and the trace is taken over the real tangent bundle.
Since $\nabla$ preserves $\omega$, it also preserves $dV$.
By \cite[Theorem~2.6(i)]{GST}, the resulting one-form is
independent of the choice of torsion-free connection
preserving $dV$, and hence of the choice of torsion-free
symplectic connection.
For a general positive volume form $\nu$, the same formula
defines $\Lambda_\nu(J,A)$ using a torsion-free connection
preserving $\nu$.
This definition applies to every infinitesimal variation $A$
of an orientation-compatible almost complex structure $J$,
that is, every endomorphism $A$ satisfying $AJ+JA=0$.
Likewise, we use
\[
 \Omega_{\nu,J}(A,B)
 =\frac12\int_M\tr(JAB)\,\nu
\]
on the space of all orientation-compatible almost complex
structures.
The identities below are understood on this larger space
whenever the vector fields involved need not preserve $\omega$.
\begin{align}
 D_J\Ric_{\nu,J}[A]&=\frac12d\Lambda_\nu(J,A),\label{eq:Ricvar}\\
 \Lambda_{e^u\nu}(J,A)&=\Lambda_\nu(J,A)+du\circ A,\label{eq:Lambdachange}\\
 \int_M\Lambda_\nu(J,A)(v)\,\nu
 &=-\Omega_{\nu,J}(A,\Lie_vJ)\label{eq:Lambdamoment}
\end{align}
For these identities, see \cite[equations~(2.11), (2.13), (2.14)]{GST}.
The sign in \eqref{eq:Lambdamoment} follows from our convention
\eqref{eq:omegaA}.
Indeed, the pointwise identity
\[
 \Lambda_\nu(J,A)(v)
 =\diver_\nu(Av)+\frac12\tr(AJ\Lie_vJ)
\]
yields \eqref{eq:Lambdamoment} upon integration over $M$.
Our symplectic form $\Omega_\nu$ is the negative of the one
used in \cite[equation~(2.6)]{GST}.
For vector fields $u$ and $v$ satisfying
$\diver_\nu u=\diver_\nu v=0$, we have
\begin{equation}\label{eq:orbitformula}
 \Omega_{\nu,J}(\Lie_uJ,\Lie_vJ)
 =-2\int_M\Ric_{\nu,J}(u,v)\,\nu.
\end{equation}
This follows from \cite[equation~(2.33)]{GST}, with the sign
adjusted to our convention for $\Omega_\nu$.
Neither $\iota_u\nu$ nor $\iota_v\nu$ is required to be exact.
For the fixed weighted volume form $d\nu=e^h dV$, set
\[
 \rho_{h,J}:=\Ric_{e^h dV,J}.
\]
The weighted Ricci form satisfies
\begin{equation}\label{eq:weightedRic}
 \rho_{h,J}=\rho_J+\frac12d(dh\circ J),\qquad
 2\Lambda_\omega\rho_{h,J}=S_J-\Delta_Jh.
\end{equation}
We also set
\begin{equation}\label{eq:w}
 w_J:=\Delta_Jh+|dh|_{g_J}^2
 =\diver_\nu(\grad_{g_J}h),\qquad
 \grad_{g_J}h=-JX_h.
\end{equation}
Our convention is $\Delta_J=\diver_{g_J}\grad_{g_J}$,
so that $\Delta_J$ has nonpositive spectrum.

We also set
\begin{equation}\label{ellJ(A)}
 \rho_J=\Ric_{dV,J},\qquad
 \ell_J(A)=\Lambda_{dV}(J,A),\qquad S_J=2\Lambda_\omega\rho_J.
\end{equation}


\section{Polar determinant and long-time phase}\label{sec:phase}

\subsection{Uniformly bounded multiplicative defect}

Let $(V,\omega_V,J_V)$ and $(W,\omega_W,J_W)$ be real
$2n$-dimensional symplectic vector spaces equipped with compatible
Hermitian structures.
Their complexifications decompose into $(1,0)$ and $(0,1)$ parts
with respect to $J_V$ and $J_W$:
\[
 V\otimes\C=V^{1,0}\oplus V^{0,1},\qquad
 W\otimes\C=W^{1,0}\oplus W^{0,1}.
\]
Let $A:V\to W$ be a real symplectic isomorphism.
Using these decompositions, write the restriction of its
complexification to $V^{1,0}$ as
\[
 A|_{V^{1,0}}=P_A+Q_A,
\]
where
\[
 P_A:V^{1,0}\to W^{1,0},\qquad
 Q_A:V^{1,0}\to W^{0,1}.
\]
Denoting their complex conjugates by $\overline{P_A}$ and
$\overline{Q_A}$, respectively, we obtain the block representation
\begin{equation}
 A=
 \begin{pmatrix}
 P_A&\overline{Q_A}\\
 Q_A&\overline{P_A}
 \end{pmatrix}.
\end{equation}
Let
\[
 \lambda_A:=\det P_A:
 \bigwedge^n V^{1,0}\longrightarrow\bigwedge^n W^{1,0},
\]
and define
\[
 \DD(A):=\frac{\lambda_A}{|\lambda_A|}:
 \bigwedge^n V^{1,0}\longrightarrow\bigwedge^n W^{1,0},
\]
where $|\lambda_A|$ is the norm induced by the Hermitian structures.

\begin{lemma}\label{lem:defect}
Let $(V_i,\omega_{V_i},J_i)$, $i=1,2,3$, be real
$2n$-dimensional symplectic vector spaces equipped with compatible
Hermitian structures, and let
\[
 A:V_1\to V_2,\qquad B:V_2\to V_3
\]
be symplectic isomorphisms.
There exists a continuous real-valued function $\delta(B,A)$
of the pair $(B,A)$ such that
\begin{equation}\label{eq:defect}
 \DD(BA)=e^{\ii\delta(B,A)}\DD(B)\DD(A),\qquad
 |\delta(B,A)|\le D_n,\qquad D_n=\frac{n\pi}{2}.
\end{equation}
Both sides of \eqref{eq:defect} are linear maps from
$\bigwedge^n V_1^{1,0}$ to $\bigwedge^n V_3^{1,0}$.

The defect may be taken to vanish for composition with the
identity and for composition of a map with its inverse.
\end{lemma}
\begin{proof}
Write $A|_{V_1^{1,0}}=P_A+Q_A$.
In unitary frames, the symplectic condition gives
$$
{}^tAJA=J,
$$
where 
$$
{}^tA=\begin{pmatrix}
P_A^*&Q_A^*\\{}^tQ_A&{}^tP_A
\end{pmatrix},\qquad 
J=
\begin{pmatrix}
i&0\\ 0&-i
\end{pmatrix}
$$
Since $A$ is invertible, we have 
$$
AJ{}^tA=J.
$$
Then one has 
\[
 P_AP_A^*-\overline{Q_A}\,\overline{Q_A}^{\,*}=I,\qquad
 P_A^*P_A-Q_A^*Q_A=I.
\]
Consequently, $P_A$ is invertible and
\begin{equation}\label{norm estimate}
 \norm{P_A^{-1}\overline{Q_A}}<1,\qquad
 \norm{Q_AP_A^{-1}}<1.
\end{equation}
Since $\DD(A)$ is given by
\[
 \DD(A)=\frac{\det_\C P_A}{|\det_\C P_A|},
\]
we have
\[
 P_{BA}=P_B(I+K)P_A,\qquad
 K=P_B^{-1}\overline{Q_B}\,Q_AP_A^{-1},\qquad
 \norm K<1.
\]
We therefore set
\[
 \delta(B,A):=\operatorname{Im}\tr_\C\Log(I+K),
\]
where $\Log$ denotes the principal matrix logarithm.
Every eigenvalue of $I+K$ lies in the open right half-plane,
and hence its argument has absolute value less than $\pi/2$.
This proves \eqref{eq:defect}.
The expression is invariant under changes of unitary frames.

Moreover, $\DD(A^{-1})=\DD(A)^{-1}$.
Thus $\delta(A^{-1},A)$ is a continuous function with values
in $2\pi\Z$. Since it vanishes at the identity and the symplectic
group is connected, it vanishes identically.
\end{proof}
This lemma gives an explicit form of the bounded defect of the
determinant quasimorphism
\cite[Example~2.6, Lemma~2.7]{CE}.
The essential point is that it bounds the defect of a lift of
the phase to $\R$, rather than resetting the angle to its
principal value at each time.
\begin{remark}[A geometric interpretation of the norm estimates]
Let $A:V_1\to V_2$ be a real symplectic isomorphism, and retain
the notation above.
We give a geometric explanation of \eqref{norm estimate}.
For $j=1,2$, let $\omega_{V_j,\C}$ denote the complex-bilinear
extension of $\omega_{V_j}$.
Define a Hermitian form on $V_j\otimes\C$ by
\[
 \mathcal H_j(z,w)
 :=-\ii\,\omega_{V_j,\C}(z,\overline w).
\]
We use the convention that Hermitian forms are complex-linear
in the first argument and conjugate-linear in the second.
With the metrics induced by
$g_j(u,v)=\omega_{V_j}(u,J_jv)$, we have
\[
 \mathcal H_j(z,z)
 =\norm{z^{1,0}}^2-\norm{z^{0,1}}^2.
\]
Thus $\mathcal H_j$ is an indefinite Hermitian form of
signature $(n,n)$.

Since $A$ is real and symplectic,
\[
 \mathcal H_2(Az,Aw)=\mathcal H_1(z,w).
\]
Also, the identity
\[
 P_A^*P_A=I+Q_A^*Q_A\ge I
\]
shows that $P_A$ is invertible.

\paragraph{The estimate $\norm{Q_AP_A^{-1}}<1$.}
Since $\mathcal H_1$ is positive definite on $V_1^{1,0}$,
the form $\mathcal H_2$ is positive definite on
$A(V_1^{1,0})$.
On the other hand,
\[
 \begin{aligned}
 A(V_1^{1,0})
 &=\{(P_Av,Q_Av):v\in V_1^{1,0}\}\\
 &=\{(u,Q_AP_A^{-1}u):u\in V_2^{1,0}\}.
 \end{aligned}
\]
Thus $A(V_1^{1,0})$ is the graph of
$T:=Q_AP_A^{-1}$.
Positivity of $\mathcal H_2$ on this graph gives
\[
 \norm{u}^2-\norm{Tu}^2>0
 \qquad (u\ne0).
\]
Hence $\norm{Tu}<\norm{u}$ for every nonzero $u$.
Since the unit sphere is compact, the function
$u\mapsto\norm{Tu}$ attains its maximum there.
It follows that
$
 \norm{Q_AP_A^{-1}}<1.
$
\paragraph{The estimate $\norm{P_A^{-1}\overline{Q_A}}<1$.}
The form $\mathcal H_1$ is negative definite on
$
 A^{-1}(V_2^{0,1}).
$
Set $C:=P_A^{-1}\overline{Q_A}$.
For $x\in V_1^{1,0}$ and $y\in V_1^{0,1}$, it follows that
$ A(x,y)\in V_2^{0,1}$ is equivalent to 
$P_Ax+\overline{Q_A}y=0$.
 Then we have  $x=-Cy.$
Therefore,
$
 A^{-1}(V_2^{0,1})
 =\{(-Cy,y):y\in V_1^{0,1}\}.
$
Negativity of $\mathcal H_1$ on this graph gives
$
 \norm{Cy}^2-\norm{y}^2<0
 \qquad (y\ne0).
$
The same compactness argument yields
$
 \norm{P_A^{-1}\overline{Q_A}}<1.
$
\end{remark}
\subsection{The canonical determinant connection}
Let $\widehat{\nabla}$ be the first canonical connection as in \eqref{the first canonical Hermitian conn}.
Let $\nabla_J^L$ denote the induced connection on its determinant
line bundle
$
 L_J:=\det T_J^{1,0}.
$
This is the standard canonical unitary connection.
Its curvature is given by the Ricci form:
\begin{equation}\label{eq:curvature}
 F_{\nabla_J^L}=-\mathrm{i}\rho_J;
\end{equation}
see Theorem~2.6(iii) of \cite{GST}.

The canonical connection $\nabla_J^L$ also admits the following
description. Fix a torsion-free symplectic connection $\nabla$ and write
\begin{equation}\label{torsion-free symplectic connection}
 \kappa_J(v)=\operatorname{tr}\bigl((\nabla J)v\bigr),
 \qquad
 \nabla_J^L
 =
 \det_{\mathbb C}\!\left(\nabla-\frac12J\nabla J\right)
 -\frac{\mathrm{i}}2\kappa_J,
\end{equation}
where $\det_{\mathbb C}(\,\cdot\,)$ denotes the induced connection
on the determinant line bundle.
The changes in the two terms cancel when the auxiliary torsion-free
symplectic connection is changed.
Indeed, if $\nabla'=\nabla+C$, the symmetry $C_uv=C_vu$ and
the identity $\operatorname{tr}C_v=0$ imply
$
 \kappa'_J(v)-\kappa_J(v)
 =-\operatorname{tr}(JC_v),
$
whereas the difference between the induced determinant connections,
evaluated on $v$, is
$
 -\frac{\mathrm{i}}2\operatorname{tr}(JC_v).
$
The same cancellation holds whenever both connections are torsion-free
and preserve $dV$.
Consequently, this construction is independent of the choice of
auxiliary connection and is natural under symplectomorphisms.
Computing its curvature yields \eqref{eq:curvature}.
When the auxiliary connection is the Levi--Civita connection of $g_J$, the almost K\"ahler identity gives $\kappa_J=0$. Hence the construction agrees with the determinant connection induced by the first canonical Hermitian connection.

Next, let $\phi_t$ be the flow of $X_h$, and let
$
 H_J(T,x):L_{x,J}\longrightarrow L_{\phi_T(x),J}
$
denote parallel transport with respect to $\nabla_J^L$ along
the trajectory $t\mapsto\phi_t(x)$.
Let
$
 P_{T,x}:=P_{D\phi_T(x)}
 :
 T^{1,0}_{x,J}M
 \longrightarrow
 T^{1,0}_{\phi_T(x),J}M
$
be the complex $(1,0)$-part of the differential
$
 D\phi_T(x):T_xM\longrightarrow T_{\phi_T(x)}M.
$
The symplectic condition implies that $P_{T,x}$ is invertible.
Its top exterior power therefore defines an isomorphism between
the corresponding fibers of the complex line bundle:
\[
 \det_{\mathbb C}P_{T,x}
 :
 L_{x,J}\longrightarrow L_{\phi_T(x),J}.
\]
Using inverse parallel transport to return to the initial fiber,
define $\lambda_J(T,x)\in\mathbb C^\times$ by
\[
 H_J(T,x)^{-1}\circ\det_{\mathbb C}P_{T,x}
 =
 \lambda_J(T,x)\operatorname{id}_{L_{x,J}}.
\]
This is a smooth complex-valued function, independent of any
trivialization of the line bundle, and it satisfies
$\lambda_J(0,x)=1$.
We define its phase by
\[
 \zeta_J(T,x)
 :=
 \frac{\lambda_J(T,x)}{|\lambda_J(T,x)|}
 \in U(1).
\]
Since $\zeta_J(0,x)=1$, lifting in the time direction determines
a unique smooth real-valued function $a_J$ satisfying
\[
 e^{\mathrm{i}a_J(T,x)}=\zeta_J(T,x),
 \qquad
 a_J(0,x)=0.
\]
Explicitly,
\[
 a_J(T,x)
 =
 \operatorname{Im}\int_0^T
 \lambda_J(t,x)^{-1}\partial_t\lambda_J(t,x)\,dt.
\]
Thus $a_J$ is globally defined on $\mathbb R\times M$.
Since $H_J(T,x)$ is unitary, $|\lambda_J(T,x)|$ equals
the norm of $\det_{\mathbb C}P_{T,x}$.
Hence the normalized determinant map is
\[
 \DD_J(D\phi_T(x))
 =
 \frac{\det_{\mathbb C}P_{T,x}}{|\lambda_J(T,x)|},
\]
and we obtain
\begin{equation}\label{eq:phase}
 \DD_J(D\phi_T(x))
 =
 e^{\mathrm{i}a_J(T,x)}H_J(T,x),
 \qquad
 a_J(0,x)=0.
\end{equation}
This is an equality of maps from $L_{x,J}$ to $L_{\phi_T(x),J}$.
The construction does not require the line bundle $L_J$ itself
to be globally trivial.
The superscript $0$ in $r_h^0$ below indicates that we use
the connection associated with the unweighted volume form $dV$.


The composition laws for the differential of the flow and for
parallel transport along the same trajectory yield the following.

\begin{proposition}
\begin{equation}\label{eq:almostadditive}
 \left|
 a_J(T_1+T_2,x)
 -a_J(T_1,x)
 -a_J(T_2,\phi_{T_1}(x))
 \right|
 \le D_n,
\end{equation}
where $D_n=\frac{n\pi}{2}$.
\end{proposition}

\begin{proof}
Parallel transport satisfies
$
 H_J(T_1+T_2,x)
 =
 H_J(T_2,\phi_{T_1}(x))\circ H_J(T_1,x).
$
Moreover, the flow identity
$
 \phi_{T_1+T_2}=\phi_{T_2}\circ\phi_{T_1}
$
implies
$
 D\phi_{T_1+T_2}(x)
 =
 D\phi_{T_2}(\phi_{T_1}(x))\circ D\phi_{T_1}(x).
$
Set
\[
 A=D\phi_{T_1}(x),
 \qquad
 B=D\phi_{T_2}(\phi_{T_1}(x)).
\]
Then
\[
 P_{BA}
 =
 P_BP_A+\overline{Q_B}Q_A
 =
 P_B(I+K)P_A,
 \qquad
 K=P_B^{-1}\overline{Q_B}Q_AP_A^{-1}.
\]
The conclusion follows from Lemma~\ref{lem:defect}.
\end{proof}
We now recall Kingman's subadditive ergodic theorem in the form
needed here.

\begin{theorem}[Kingman's subadditive ergodic theorem {\cite{Kingman}}]
Let $(X,\mu)$ be a probability space, and let $S:X\to X$
be a measure-preserving transformation.
Suppose that real-valued measurable functions $b_m$, $m\ge1$,
satisfy
$
 b_{m+k}(x)\le b_m(x)+b_k(S^mx)
$
and that, for some constant $C>0$,
$
 |b_m(x)|\le Cm
$
almost everywhere for every $m\ge1$.
Then
\[
 \lim_{m\to\infty}\frac{b_m(x)}m=b_\infty(x)
\]
almost everywhere.
The limit is finite and $S$-invariant.
Under the stated uniform bound, the dominated convergence theorem
also gives convergence in $L^1$.
\end{theorem}

Applying Kingman's subadditive ergodic theorem yields the following.

\begin{proposition}\label{prop:asymptotic}
The limit
\begin{equation}\label{eq:rdef}
 r_h^0(J)(x)
 =
 \lim_{T\to\infty}\frac{a_J(T,x)}T
\end{equation}
exists almost everywhere.
The function $r_h^0(J)$ belongs to $L^\infty(M,d\nu)$ and
is invariant under $\phi_t$.
The convergence holds in $L^p$ for every $1\le p<\infty$,
and also in the weak-$*$ topology of $L^\infty$.
\end{proposition}

\begin{proof}
Iterating \eqref{eq:almostadditive} at integer times gives
\[
 |a_J(m,x)|
 \le
 m\lVert a_J(1,\cdot)\rVert_\infty+(m-1)D_n.
\]
Thus $a_J(m,\cdot)/m$ is uniformly bounded.
Set
$
 b_m=a_J(m,\cdot)+D_n.
$
Then we obtain
$
 b_{m+k}\le b_m+b_k\circ\phi_m.
$
We may therefore apply Kingman's subadditive ergodic theorem
to the measure-preserving transformation $\phi_1$ by using the normalized measure 
$d\nu/\text{\rm Vol}_{d\nu}(M)$.
The required integrability and the lower bound $|b_m|\leq Cm$ follow
from the estimate above
(see also \cite[Section~3.1]{CE} for this type of application).
Hence a finite limit exists along integer times.

Write $T=m+\tau$, where $m$ is an integer and $0\le\tau<1$.
By \eqref{eq:almostadditive},
\[
 |a_J(T,x)-a_J(m,x)|
 \le
 \sup_{[0,1]\times M}|a_J|+D_n.
\]
The limit therefore extends to all real times $T\to\infty$.
Uniform boundedness and the dominated convergence theorem
give convergence in $L^p$ for every $1\le p<\infty$.
For any $\psi\in L^1$, the estimate
\[
 \left|\psi\,\frac{a_J(T)}T\right|
 \le C|\psi|
\]
together with uniform boundedness and the dominated convergence theorem gives weak-$*$ convergence.

For fixed $\tau$, decomposing the time $T+\tau$ in two different
orders yields
\[
 |a_J(T,\phi_\tau(x))-a_J(T,x)|
 \le
 2\lVert a_J(\tau,\cdot)\rVert_\infty+2D_n.
\]
Dividing by $T$ shows that the limit is invariant under $\phi_\tau$.
This invariance is understood at the level of equivalence classes
in $L^\infty$.
\end{proof}

\begin{remark}
The Banach--Alaoglu theorem alone only yields a weak-$*$ convergent
subsequence.
Proposition~\ref{prop:asymptotic} additionally uses almost additivity
and the subadditive ergodic theorem to obtain a limit as
$T\to\infty$ without passing to a subsequence.
It does not, however, assert convergence in the $L^\infty$ norm
or smoothness of the limit.
\end{remark}
\section{A continuous coboundary representation of the Ricci cocycle}
\label{sec:cocycle}

For a closed $2$-form $\eta$,  we define
\[
 C_{\eta,h}(f,g)
 =
 \int_M\eta(X_f,X_g)\dn,
 \qquad f,g\in\gh.
\]
This is a continuous $2$-cocycle.
We use the convention
$
 \delta B(f,g)=-B(\{f,g\})
$
for the Chevalley--Eilenberg differential.

\begin{lemma}\label{lem:weightedwedge}
For $f,g\in\gh$, we have
\begin{equation}\label{eq:weightedwedge}
 C_{\eta,h}(f,g)
 =
 \int_M(\Lambda_\omega\eta)\{f,g\}\dn
 -
 \int_M f\,\eta(X_h,X_g)\dn.
\end{equation}
\end{lemma}

\begin{proof}
For $n\ge2$, expanding the double contraction gives
\[
 \eta(X_f,X_g)dV
 =
 (\Lambda_\omega\eta)\{f,g\}dV
 -
 df\wedge dg\wedge\eta\wedge
 \frac{\omega^{n-2}}{(n-2)!}.
\]
Multiply the last term by $e^h$, integrate, and integrate by parts
using $d\eta=0$.
Since $\{h,g\}=0$, we have
\[
 dh\wedge dg\wedge\eta\wedge
 \frac{\omega^{n-2}}{(n-2)!}
 =
 -\eta(X_h,X_g)dV.
\]
This yields \eqref{eq:weightedwedge}.
For $n=1$, the same identity follows from
\[
 \eta=(\Lambda_\omega\eta)\omega,
 \qquad
 \eta(X_h,X_g)=0.
\]
\end{proof}
Let $\psi_s$ be the flow of $X_g$.
Since $\{h,g\}=0$, the flows commute:
$
 \phi_T\psi_s=\psi_s\phi_T.
$
Their differentials satisfy
\[
 D\psi_s(\phi_Tx)D\phi_T(x)
 =
 D\phi_T(\psi_sx)D\psi_s(x).
\]
Apply $\DD_J$ to both sides and use Lemma~\ref{lem:defect}.
The ratio of parallel transports along the two paths is the
holonomy around the rectangle parametrized by
\[
 R(t,\sigma)=\phi_t\psi_\sigma x.
\]
With the curvature convention \eqref{eq:curvature} and the
orientation $dt\wedge d\sigma$, we obtain
\[
 H_g(s,\phi_Tx)H_h(T,x)
 =
 e^{\mathrm{i}I_{T,s}(x)}
 H_h(T,\psi_sx)H_g(s,x),
\]
where
\begin{equation}\label{ITs(x)}
 I_{T,s}(x)
 =
 \int_0^T\int_0^s
 \rho_J(X_h,X_g)(\phi_t\psi_\sigma x)\,d\sigma\,dt.
\end{equation}
We then have the following.

\begin{proposition}
\begin{equation}\label{eq:rectanglephase}
 a_J(T,\psi_sx)-a_J(T,x)
 =
 I_{T,s}(x)+E_{T,s}(x),
\end{equation}
where
\[
 \lVert E_{T,s}\rVert_\infty
 \le
 2\lVert a_{g,J}(s,\cdot)\rVert_\infty+2D_n.
\]
Here $a_{g,J}$ is the analogous relative phase associated with
the flow of $X_g$.
The passage from an equality of exponentials to an equality
of real-valued phases follows from continuity and agreement
at $T=0$.
For fixed $s$, the error bound is independent of $T$.
\end{proposition}

\begin{proof}
In this proof, we write $a_{h,J}$ for the previously defined
phase $a_J$, and use the separate notation $a_{g,J}$ for the
phase associated with $\psi_s$.

Let $(M,\omega)$ be a closed symplectic manifold, and fix an
$\omega$-compatible almost complex structure $J$.
Let
\[
 \phi_t=\operatorname{Fl}^{X_h}_t,
 \qquad
 \psi_s=\operatorname{Fl}^{X_g}_s
\]
be the flows of Hamiltonians $h,g$ satisfying $\{h,g\}=0$.
These flows commute:
$
 \phi_t\psi_s=\psi_s\phi_t.
$
We use the curvature convention
$
 F_{\nabla_J^L}=-\mathrm{i}\rho_J
$
for the canonical connection $\nabla_J^L$ on
$L_J=\det_{\mathbb C}T_J^{1,0}M$.
Denote parallel transport along the respective trajectories
by $H_h$ and $H_g$, and define the real-valued phases by
\begin{equation}\label{eq:two-phases}
 \begin{aligned}
  \DD_J(D\phi_T(x))
   &=e^{\mathrm{i}a_{h,J}(T,x)}H_h(T,x),\\
  \DD_J(D\psi_s(x))
   &=e^{\mathrm{i}a_{g,J}(s,x)}H_g(s,x).
 \end{aligned}
\end{equation}
Both phases are continuous lifts normalized to vanish at time $0$.
Here $\DD_J$ is the polar determinant, and these are equalities
of maps between the determinant lines at the initial and
terminal points.
Note that {the phase $a_{h,J}$ is exactly the previously defined
$a_J$, whereas $a_{g,J}$ is an auxiliary phase introduced
for this proof.}
We occasionally suppress the fixed subscript $J$ below.


Consider the two paths
\[
 \begin{aligned}
  \gamma_1:\quad&
  x\xrightarrow{\ \phi_T\ }\phi_Tx
   \xrightarrow{\ \psi_s\ }\psi_s\phi_Tx,\\
  \gamma_2:\quad&
  x\xrightarrow{\ \psi_s\ }\psi_sx
   \xrightarrow{\ \phi_T\ }\phi_T\psi_sx.
 \end{aligned}
\]
Their terminal points agree because the flows commute.
Since we let $T\to\infty$ while keeping $s$ fixed, we refer
to the edges along $\phi$ as the long edges and those along
$\psi$ as the short edges.

\medskip
\medskip

\begin{tikzpicture}[
  >=stealth,
  pathone/.style={->,thick,shorten <=2pt,shorten >=2pt},
  pathtwo/.style={pathone,dashed},
  every node/.style={font=\normalsize}
]
  \coordinate (x) at (0,0);
  \coordinate (p) at (8,0);
  \coordinate (q) at (0,2.4);
  \coordinate (r) at (8,2.4);

  \draw[pathone] (x) -- node[below=5pt]
    {$\phi_T$\; (long edge)} (p);
  \draw[pathone] (p) -- node[right=6pt,align=left]
    {$\psi_s$\\[-1pt]\small (short edge)} (r);

  \draw[pathtwo] (x) -- node[left=6pt,align=right]
    {$\psi_s$\\[-1pt]\small (short edge)} (q);
  \draw[pathtwo] (q) -- node[above=5pt]
    {$\phi_T$\; (long edge)} (r);

  \foreach \v in {x,p,q,r} \fill (\v) circle (1.5pt);
  \node[below left=4pt] at (x) {$x$};
  \node[below right=4pt] at (p) {$\phi_T x$};
  \node[above=7pt] at (q) {$\psi_s x$};
  \node[above=7pt] at (r)
    {$\psi_s\phi_T x=\phi_T\psi_s x$};

  \draw[->,thick] (1.2,-1.05) -- (2.1,-1.05);
  \node[right=4pt] at (2.1,-1.05) {$\gamma_1$};
  \draw[->,thick,dashed] (5.0,-1.05) -- (5.9,-1.05);
  \node[right=4pt] at (5.9,-1.05) {$\gamma_2$};
  \node at (4,-1.65) {$T\to\infty,\qquad s\text{ fixed}$};
\end{tikzpicture}

The quantity to be compared is the phase difference between
the two long edges:
\[
 a_h(T,\psi_sx)-a_h(T,x).
\]

Set
\[
 \begin{aligned}
  H_1&:=H_g(s,\phi_Tx)\,H_h(T,x),\\
  H_2&:=H_h(T,\psi_sx)\,H_g(s,x).
 \end{aligned}
\]
These are the parallel transports along the two paths.
Both are unitary isomorphisms from $L_{J,x}$ to the
determinant line at their common terminal point.


First suppose that $T,s>0$, and consider the map
\[
 R:[0,T]\times[0,s]\longrightarrow M,
 \qquad
 R(t,\sigma)=\phi_t\psi_\sigma x.
\]
Since the parameter rectangle is contractible, the pullback
line bundle $R^*L_J$ is trivial.
We may therefore choose a unitary frame and write the
pullback connection as
\[
 R^*\nabla_J^L=d-\mathrm{i}\lambda,
\]
where $\lambda$ is a real-valued $1$-form on the rectangle.
The curvature convention gives
\[
 -\mathrm{i}\,d\lambda
 =
 R^*F_{\nabla_J^L}
 =
 -\mathrm{i}R^*\rho_J,
 \qquad
 d\lambda=R^*\rho_J.
\]
No global triviality of $L_J$ itself is required.
Nor does $R$ need to be an embedding, since the argument
takes place on the pullback line bundle.

In this frame, the coefficient $u$ of a parallel section
satisfies
\[
 du-\mathrm{i}\lambda u=0,
\]
so parallel transport is given by
$\exp(\mathrm{i}\int\lambda)$.
The loop obtained by following $\gamma_1$ and then traversing
$\gamma_2$ in reverse corresponds to the positively oriented
boundary of the rectangle with respect to $dt\wedge d\sigma$.
Stokes' theorem therefore gives
\[
 \begin{aligned}
 H_2^{-1}H_1
 &=
 \exp\!\left(
  \mathrm{i}
  \int_{\partial([0,T]\times[0,s])}\lambda
 \right)\operatorname{id}\\
 &=
 \exp\!\left(
  \mathrm{i}
  \int_{[0,T]\times[0,s]}R^*\rho_J
 \right)\operatorname{id}.
 \end{aligned}
\]
Since the flows commute,
\[
 R_*\partial_t=X_h,
 \qquad
 R_*\partial_\sigma=X_g.
\]
Thus, setting
\[
 I_{T,s}(x)
 :=
 \int_0^T\int_0^s
 \rho_J(X_h,X_g)(\phi_t\psi_\sigma x)\,d\sigma\,dt,
\]
we obtain
\begin{equation}\label{eq:holonomy}
 {H_1=e^{\mathrm{i}I_{T,s}(x)}H_2.}
\end{equation}
This is the exact identity comparing the two parallel transports.
The same formula holds for $s<0$ when the integrals are
interpreted as oriented integrals, and it extends to $T=0$
or $s=0$ by continuity.


Let
\[
 \begin{aligned}
  A&:=D\phi_T(x),&
  B&:=D\psi_s(\phi_Tx),\\
  C&:=D\psi_s(x),&
  D&:=D\phi_T(\psi_sx)
 \end{aligned}
\]
be the differentials corresponding to the four edges.
Differentiating the commutation identity for the flows gives
\begin{equation}\label{eq:derivative-square}
 BA=DC.
\end{equation}
Here $A,B,C,D$ are symplectic isomorphisms between the tangent
spaces at the initial and terminal points of their
respective edges.

Lemma~\ref{lem:defect} provides, for composable symplectic
isomorphisms $A,B$, a continuous real-valued multiplicative
defect $\delta(B,A)$ satisfying
\[
 \DD_J(BA)
 =
 e^{\mathrm{i}\delta(B,A)}\DD_J(B)\DD_J(A),
 \qquad
 |\delta(B,A)|\le D_n.
\]
The defect vanishes when either factor is the identity.
Apply this lemma to the two paths and write
\[
 \delta_1(T,s,x):=\delta(B,A),
 \qquad
 \delta_2(T,s,x):=\delta(D,C).
\]
Thus $|\delta_1|,|\delta_2|\le D_n$.
Substituting the phase identities \eqref{eq:two-phases},
we obtain
\begin{equation}\label{eq:composed-phases}
 \begin{aligned}
  \DD_J(BA)
   &=
   e^{\mathrm{i}\{\delta_1+a_g(s,\phi_Tx)+a_h(T,x)\}}H_1,\\
  \DD_J(DC)
   &=
   e^{\mathrm{i}\{\delta_2+a_h(T,\psi_sx)+a_g(s,x)\}}H_2.
 \end{aligned}
\end{equation}


Using \eqref{eq:holonomy} and \eqref{eq:derivative-square}
in \eqref{eq:composed-phases}, and cancelling the common
parallel transport $H_2$, gives
\begin{equation}\label{eq:exponentials}
 \begin{aligned}
 &\exp\!\left(
  \mathrm{i}\{\delta_1+a_g(s,\phi_Tx)+a_h(T,x)+I_{T,s}(x)\}
 \right)\\
 &\qquad=
 \exp\!\left(
  \mathrm{i}\{\delta_2+a_h(T,\psi_sx)+a_g(s,x)\}
 \right).
 \end{aligned}
\end{equation}

At this stage, we only know that the difference between
the real-valued phases in the two exponentials belongs
to $2\pi\mathbb Z$.
For fixed $s,x$, however, this difference is continuous in $T$.
At $T=0$, we have
\[
 a_h(0,\cdot)=0,
 \qquad
 I_{0,s}=0,
 \qquad
 \delta_1(0,s,x)=\delta_2(0,s,x)=0,
\]
and
\[
 a_g(s,\phi_0x)=a_g(s,x).
\]
Hence the difference vanishes at $T=0$.
A continuous $2\pi\mathbb Z$-valued function on a connected
interval is constant, so the difference vanishes for
every $T\ge0$.
We therefore obtain the following equality of real numbers:
\begin{equation}\label{eq:expanded-error}
 \begin{aligned}
  a_h(T,\psi_sx)-a_h(T,x)
   ={}&I_{T,s}(x)+a_g(s,\phi_Tx)-a_g(s,x)\\
      &+\delta_1(T,s,x)-\delta_2(T,s,x).
 \end{aligned}
\end{equation}
This is the expanded form of \eqref{eq:rectanglephase},
with all error terms written explicitly.


Restoring the subscript $J$, define
\[
 \begin{aligned}
 E_{T,s}(x):={}&a_{g,J}(s,\phi_Tx)-a_{g,J}(s,x)\\
              &+\delta_1(T,s,x)-\delta_2(T,s,x).
 \end{aligned}
\]
Then \eqref{eq:expanded-error} becomes
\begin{equation}\label{eq:rectanglephase-restated}
 a_{h,J}(T,\psi_sx)-a_{h,J}(T,x)
 =
 I_{T,s}(x)+E_{T,s}(x).
\end{equation}

For fixed $s$, the function $a_{g,J}(s,\cdot)$ is smooth
on the compact manifold $M$ and is therefore bounded.
Consequently,
\[
 \lVert E_{T,s}\rVert_\infty
 \le
 2\lVert a_{g,J}(s,\cdot)\rVert_\infty+2D_n.
\]
The right-hand side is independent of $T$.
Thus, for fixed $s$,
\[
 \lVert E_{T,s}/T\rVert_\infty\longrightarrow0
 \qquad\text{as }T\to\infty.
\]
The error remains bounded because the time $s$ along
the short edges is fixed, so their phases remain bounded
as the initial point varies, and because both multiplicative
defects are uniformly bounded.
\end{proof}

\begin{proposition}[The rectangle of commuting flows]
\label{prop:rectangle}
The function $r=r_h^0(J)$ satisfies
\begin{equation}\label{eq:residual}
 \int_M r\,\{f,g\}\dn
 =
 \int_M f\,\rho_J(X_h,X_g)\dn,
 \qquad f,g\in\gh.
\end{equation}
\end{proposition}

\begin{proof}
Multiplying \eqref{eq:rectanglephase} by $f$, integrating over $M$,
and dividing by $T$ gives  
\begin{equation}\label{frac1TintMf}
\frac1T\int_M f (a_J(T,\psi_sx)-a_J(T,x) ) d\nu
 =
\frac1T\int_Mf(x)( I_{T,s}(x)+E_{T,s}(x))d\nu.
\end{equation}

Using $f\circ\phi_t=f$ and the fact that the flow $\psi_s$ 
preserves the measure $d\nu$, the change of variables $y=\psi_s(x)$ gives
\begin{align*}
\frac1T\int_M f(x) a_J(T,\psi_sx)d\nu_x=&\frac1T\int_M f(\psi_{-s}(y)) a_J(T, y) d\nu_{y}
\end{align*}
Thus the left-hand side of \eqref{frac1TintMf}
 is 
 $$
  \int_M(f\circ\psi_{-s}-f)\frac{a_J(T)}T\dn.
 $$

By \eqref{ITs(x)}, we have 
$$ 
\frac1T\int_M f(x)I_{T,s}(x)d\nu_x
 =
\frac1T\int_M \left( \int_0^T\int_0^s
 f(x)\rho_J(X_h,X_g)(\phi_t\psi_\sigma x)\,d\sigma\,dt \right)d\nu_x
$$
For each $t$ and $\sigma$, use the commutativity of the flows $\phi_t\psi_\sigma x=\psi_\sigma \phi_t x$,
and set $y=\phi_t(x)$.
Since $f\circ\phi_t=f$ and both flows $\phi_t$ and $\psi_s$
preserve the measure,  we have 
\begin{align*}
\frac1T\int_M f(x)I_{T,s}(x)d\nu_x
 =&
\frac1T\int_M \left( \int_0^T\int_0^s
 f(\phi_t^{-1}y)\rho_J(X_h,X_g)(\psi_\sigma y)\,d\sigma\,dt \right)d\nu_{y}\\
 =&\frac1T\int_M \left( \int_0^T\int_0^s
 f(y)\rho_J(X_h,X_g)(\psi_\sigma y )\,d\sigma\,dt \right)d\nu_{y}\\
 =&\frac1T\int_M \left( \int_0^T\int_0^s
( f\circ\psi_{-\sigma}(z))\rho_J(X_h,X_g)_z\,d\sigma\,dt \right)d\nu_z
\end{align*}
In the last equality, we used the change of variables
$z=\psi_\sigma(y)$ and the invariance of $d\nu$ under $\psi_\sigma$.
The integrand in the last expression is independent of $t$.
Consequently,
we obtain
$$
\frac1T\int_M f(x)I_{T,s}(x)d\nu_x= \int_0^s\int_M(f\circ\psi_{-\sigma})\rho_J(X_h,X_g)\dn\,d\sigma
$$
For fixed $s$, the error term $E_{T,s}(x)$ is uniformly bounded
in $T$ and $x$.
Hence
\[
 \int_M(f\circ\psi_{-s}-f)\frac{a_J(T)}T\dn
 =
 \int_0^s\int_M(f\circ\psi_{-\sigma})\rho_J(X_h,X_g)\dn\,d\sigma
 +O(T^{-1}).
\]
Pass to the limit using Proposition~\ref{prop:asymptotic},
and then differentiate at $s=0$.
Since
\[
 \left.\frac{\partial}{\partial s}\right|_{s=0}
 (f\circ\psi_{-s})
 =
 -X_gf
 =
 \{f,g\},
\]
we obtain 
\[
 \int_M r\,\{f,g\}\dn
 =
 \int_M f\,\rho_J(X_h,X_g)\dn,
\]
which proves \eqref{eq:residual}.

\end{proof}

\begin{theorem}[Vanishing of the Ricci cocycle class]
\label{thm:riccicocycle}
\begin{equation}\label{eq:Ricprimitive}
 C_{\rho_J,h}(f,g)
 =
 \int_M
 \left(\frac{S_J}{2}-r_h^0(J)\right)
 \{f,g\}\dn.
\end{equation}
Consequently,
\[
 [C_{\rho_J,h}]=0
 \quad\text{in }\Hcont(\gh,\mathbb R).
\]
\end{theorem}

\begin{proof}
Combine Lemma~\ref{lem:weightedwedge} with
Proposition~\ref{prop:rectangle}.
Define
\[
 B_J(f)
 =
 \int_M f\left(\frac{S_J}{2}-r_h^0(J)\right)\dn.
\]
Then
\[
 C_{\rho_J,h}=-\delta B_J.
\]
Moreover, the estimate
\[
 |B_J(f)|
 \le
 \left\lVert\frac{S_J}{2}-r_h^0(J)\right\rVert_\infty
 \nu(M)\lVert f\rVert_{C^0}
\]
shows that $B_J$ is continuous with respect to the
Fr\'echet topology.
\end{proof}

Theorem~\ref{thm:riccicocycle} gives an explicit continuous
coboundary representation.
Algebraic vanishing on finite sums of brackets alone does not
immediately imply this conclusion in continuous cohomology.
Furthermore, \eqref{eq:Ricprimitive} alone does not determine
the variation with respect to $J$.
The comparison in the next section is needed to obtain
an explicit moment map.
\section{Variation of almost complex structures and comparison of phases}\label{sec:variation}

\subsection{Parallel transport in the parameter direction and mixed curvature}
Let $(M,\omega)$ be a symplectic manifold of real dimension $2n$, let
$dV=\omega^n/n!$, and let $s\mapsto J_s$ be a smooth path of
$\omega$-compatible almost complex structures, with $s$ in an interval
$I\subset\R$ containing $0$.
Fix a torsion-free connection $\nabla$ on $TM$ preserving $\omega$.
On the pullback of $TM$ to $I\times M$, we define a connection $D$ by
\[
 D_v=\nabla_v,\qquad D_{\partial_s}=\partial_s, 
\]
where 
$v\in\Gamma(TM)$ is identified with its $s$-independent lift
to $I\times M$.
Set
\[
 \widetilde D:=D-\frac12J(DJ).
\]
Here $J$ denotes the endomorphism whose restriction to $\{s\}\times M$
is $J_s$.
 On the family of complex determinant lines
$L_{J_s}=\bigwedge^nT^{1,0}_{J_s}M$, we consider the connection
\[
\mathcal D
=\det_{\C}(\widetilde D)-\frac{\ii}{2}\kappa_J=
\det_{\mathbb C}\!\left(D-\frac12J(DJ)\right)
-\frac{i}{2}\kappa_J,
\]
where the $1$-form $\kappa_J(v)=\tr\bigl((\nabla J)v\bigr)$
 is extended to $I\times M$ by setting
$\kappa_J(\partial_s)=0$, see \eqref{torsion-free symplectic connection}, and 
$\det_{\C}(\widetilde D)$ denotes the induced connection on the
determinant line bundle. 
The restriction of $\mathcal D$ to the $M$-directions is the
canonical connection $\nabla^L_{J_s}$.
\begin{lemma}[Mixed curvature]\label{lem:mixed}
Recall that the one-form $\ell_J(A)=\Lambda_{dV}(J,A)$ is given by
\[
 \ell_J(A)(v)
 =\tr\bigl((\nabla A)v\bigr)
  +\frac12\tr\bigl(AJ\nabla_vJ\bigr).
\]
Then one has 
\begin{equation}\label{eq:mixed}
iF^{\mathcal D}(\partial_s,v)
=
\frac12\ell_{J_s}(\dot J_s)(v),
\end{equation}
for every $v\in TM$.
\end{lemma}

\begin{proof}
At the point under consideration, abbreviate
$
 J=J_s,\quad A=\dot J_s,\quad P=\nabla_vJ_s.
$
Differentiating $J^2=-\id$ in the $s$- and $v$-directions gives
\begin{equation}\label{eq:anticommutation}
 JA=-AJ,\qquad JP=-PJ.
\end{equation}
Since curvature is tensorial, we may extend $v$ to a vector field
independent of $s$. Then one has
$
 [\partial_s,v]=0.
$
Moreover, since $\nabla$ is independent of $s$, we have 
$ R^D(\partial_s,v)=0, $ and  
$\partial_s(\nabla_vJ_s)=\nabla_vA.$
Write the difference between the connections as
$
 Q:=-\frac12J(DJ).
$
Then we have
$
 Q(\partial_s)=-\frac12JA,
$
 and $ Q(v)=-\frac12JP.
$
The curvature transformation formula therefore gives
\begin{equation}\label{eq:curvature-change}
 R^{\widetilde D}(\partial_s,v)
 =\partial_s\bigl(Q(v)\bigr)
  -\nabla_v\bigl(Q(\partial_s)\bigr)
  +[Q(\partial_s),Q(v)].
\end{equation}
Here $\nabla_v$ acts on endomorphisms by the induced connection.
The first two terms are
\[
 \begin{aligned}
  \partial_s\bigl(Q(v)\bigr)
  &=-\frac12\bigl(AP+J\nabla_vA\bigr),\\
  \nabla_v\bigl(Q(\partial_s)\bigr)
  &=-\frac12\bigl(PA+J\nabla_vA\bigr).
 \end{aligned}
\]
Subtracting them cancels the terms involving $\nabla_vA$ and yields
\begin{equation}\label{eq:derivative-difference}
 \partial_s\bigl(Q(v)\bigr)
 -\nabla_v\bigl(Q(\partial_s)\bigr)
 =-\frac12[A,P].
\end{equation}
On the other hand, \eqref{eq:anticommutation} implies $JAJ=A$ and
$JPJ=P$, so
\begin{equation}\label{eq:commutator-term}
 [Q(\partial_s),Q(v)]
 =\frac14[JA,JP]
 =\frac14[A,P].
\end{equation}
Adding \eqref{eq:derivative-difference} and \eqref{eq:commutator-term},
we obtain
\begin{equation}\label{eq:mixed-tangent}
 {
 R^{\widetilde D}(\partial_s,v)
 =-\frac14[A,P]
 =-\frac14[\dot J_s,\nabla_vJ_s].
 }
\end{equation}
The curvature of the induced connection on the determinant line is the
complex trace of the curvature of the original connection. 
We denote by $\tr_\C$ the complex trace on $T^{1,0}_{J_s}M$.  Thus
\[
 F^{\det_{\C}\widetilde D}(\partial_s,v)
 =\tr_{\C}\bigl(R^{\widetilde D}(\partial_s,v)\bigr).
\]
For a complex-linear real endomorphism $T$ of $(T_xM,J)$,
\begin{equation}\label{eq:complex-trace}
 \tr_{\C}T
 =\frac12\tr T-\frac{\ii}{2}\tr(JT).
\end{equation}
Indeed, $J$ acts as multiplication by $\ii$, and
\[
 \tr T=2\operatorname{Re}\tr_{\C}T,
 \qquad
 \tr(JT)=-2\operatorname{Im}\tr_{\C}T.
\]

The connection $\widetilde D$ preserves $J$. To see this, put
$E=D_XJ$ for any direction $X$. Then $EJ+JE=0$ and $JEJ=E$, whence
\[
 \widetilde D_XJ
 =E-\frac12[JE,J]
 =E-\frac12(E+E)=0.
\]
Consequently, its curvature
$T=R^{\widetilde D}(\partial_s,v)$ is complex-linear.
By \eqref{eq:mixed-tangent}, it also has zero real trace:
\[
 \tr T=-\frac14\tr[A,P]=0.
\]
This argument requires no assumption that $\widetilde D$ is a Hermitian
connection on the tangent bundle. If the auxiliary connection is
symplectic, then $D\omega=0$ and $\widetilde D$ is indeed Hermitian;
in that case, the vanishing of the real trace also follows from the
skew-Hermitian character of its curvature.

Equations \eqref{eq:mixed-tangent} and \eqref{eq:complex-trace} give
\[
 F^{\det_{\C}\widetilde D}(\partial_s,v)
 =\frac{\ii}{8}\tr\bigl(J[A,P]\bigr).
\]
Using the anticommutation relations and cyclicity of the trace, we find
\[
 \begin{aligned}
  \tr\bigl(J[A,P]\bigr)
  &=\tr(JAP)-\tr(JPA)\\
  &=-\tr(AJP)-\tr(AJP)\\
  &=-2\tr(AJP).
 \end{aligned}
\]
Therefore,
\begin{equation}\label{eq:mixed-determinant}
 {
 F^{\det_{\C}\widetilde D}(\partial_s,v)
 =-\frac{\ii}{4}\tr\bigl(\dot J_sJ_s\nabla_vJ_s\bigr).
 }
\end{equation}
Adding a scalar-valued one-form $\kappa_J$ to a line-bundle connection adds its
exterior derivative to the curvature. Hence
\begin{equation}\label{eq:corrected-curvature}
 F^{\mathcal D}
 =F^{\det_{\C}\widetilde D}-\frac{\ii}{2}d\kappa_J.
\end{equation}
Here $d$ is the exterior derivative on $I\times M$, including the
$s$-direction.

Since $\kappa_J(\partial_s)=0$ and $[\partial_s,v]=0$,
\begin{equation}\label{eq:mixed-kappa}
 \begin{aligned}
  d\kappa_J(\partial_s,v)
  &=\partial_s\bigl(\kappa_J(v)\bigr)
    -v\bigl(\kappa_J(\partial_s)\bigr)
    -\kappa_J([\partial_s,v])\\
  &=\partial_s\bigl(\kappa_J(v)\bigr)\\
  &=\tr\bigl((\nabla\dot J_s)v\bigr).
 \end{aligned}
\end{equation}
Thus, although $\kappa_J$ has no component in the $s$-direction, its
components in the $M$-directions depend on $s$ and therefore contribute
to the mixed curvature.

Combining \eqref{eq:mixed-determinant}--\eqref{eq:mixed-kappa} gives
\[
 \begin{aligned}
  \ii F^{\mathcal D}(\partial_s,v)
  &=\frac14\tr\bigl(\dot J_sJ_s\nabla_vJ_s\bigr)
    +\frac12\tr\bigl((\nabla\dot J_s)v\bigr)\\
  &=\frac12\left[
    \tr\bigl((\nabla\dot J_s)v\bigr)
    +\frac12\tr\bigl(\dot J_sJ_s\nabla_vJ_s\bigr)
    \right]\\
  &=
  {\frac12\ell_{J_s}(\dot J_s)(v)}.
 \end{aligned}
\]
This proves the mixed-curvature identity.
\end{proof}
Define $B_s:TM\to TM$ by
\begin{equation}\label{eq:B}
\partial_s B_s
=
\frac12J_s\dot J_sB_s,
\qquad
B_0=\operatorname{id}.
\end{equation}
Since
$
 \widetilde D_{\partial_s}
 =\partial_s-\frac12J_s\dot J_s,
$
for every fixed $u\in T_xM$ we see that
the condition 
 $\widetilde D_{\partial_s}(B_su)=0$
is equivalent to 
$\partial_sB_s=\frac12J_s\dot J_sB_s.$
Thus $B_s$ is parallel transport for $\widetilde D$ in the parameter
direction. 
Since $D\omega=0$, the connection $\widetilde D$ preserves both
$\omega$ and $J$. Consequently, $B_s$ is symplectic and satisfies
$
J_sB_s=B_sJ_0.
$
Since $\widetilde DJ=0$, it is complex-linear as a map
\[
 B_s:(T_xM,J_0)\longrightarrow(T_xM,J_s).
\]
Hence
the map
$
I_s=\det_{\mathbb C}B_s:
L_{J_0}\longrightarrow L_{J_s}
$
is the parallel transport of $\mathcal D$ in the $s$-direction.
Under this identification, we obtain
\begin{equation}\label{eq:beta}
I_s^{-1}\nabla_{J_s}^LI_s
=
\nabla_{J_0}^L-i\beta_s,
\qquad
\beta_s
=
\frac12\int_0^s
\ell_{J_\tau}(\dot J_\tau)\,d\tau.
\end{equation}
Indeed, the $s$-derivative of the connection one-form on the left-hand side is given by the mixed curvature, and the formula follows by integrating \eqref{eq:mixed}.
Consequently, after identifying the determinant line bundles via
$I_s$, the parallel transports satisfy
\begin{equation}\label{eq:transportcomparison}
I_s(\phi_Tx)^{-1}
H_{J_s}(T,x)
I_s(x)
=
\exp\!\left(
i\int_0^T
\beta_s(X_h)(\phi_tx)\,dt
\right)
H_{J_0}(T,x).
\end{equation}
This is an exact identity for the difference of the connections; in particular, no bounded error term occurs here.

\subsection{Comparison error for the polar determinant}
We now compare separately the phase of parallel transport and the polar-determinant phase of the differential of the flow.
Write
$
y=\phi_Tx
$ and 
$
A=D\phi_T(x).
$
Transporting the $J_s$-polar determinant by the unitary map $I_s$, we obtain
\[
I_s(y)^{-1}
\mathscr D_{J_s}(A)
I_s(x)
=
\mathscr D_{J_0}
\bigl(B_s(y)^{-1}AB_s(x)\bigr).
\]
The right-hand side is measured using $J_0$ at both the initial and terminal points.
For the symplectic endomorphism $B_s(p)$, write
\[
\mathscr D_{J_0}(B_s(p))
=
e^{ib_s(p)},
\]
where the real-valued phase is chosen continuously starting from $b_0=0$.
Then $b_s$ is a smooth function on $M$.
Applying Lemma \ref{lem:defect} twice yields
\begin{align}
&\mathscr D_{J_0}
\bigl(B_s(y)^{-1}AB_s(x)\bigr)
=
e^{iE_{s,T}(x)}
\mathscr D_{J_0}(A),
\label{eq:polarcomparison}\\
&E_{s,T}(x)
=
b_s(x)-b_s(y)
+\delta\bigl(B_s(y)^{-1},AB_s(x)\bigr)
+\delta\bigl(A,B_s(x)\bigr),
\notag\\
&\|E_{s,T}\|_\infty
\le
2\|b_s\|_\infty+2D_n.
\label{eq:polarerror}
\end{align}
Notice that $\|D\phi_T\|$ does not appear in this estimate.
Thus, even when $D\phi_T$ grows with $T$, the error remains uniformly bounded in $T$ for each fixed $s$.
Moreover, $E_{s,0}=0$, and the real-valued phase is chosen continuously starting from $T=0$.
Substituting \eqref{eq:transportcomparison} and \eqref{eq:polarcomparison} into \eqref{eq:phase}, we obtain
\begin{equation}\label{eq:phasecomparison}
a_{J_s}(T,x)-a_{J_0}(T,x)
=
-\int_0^T
\beta_s(X_h)(\phi_tx)\,dt
+E_{s,T}(x).
\end{equation}
The phase difference arising from parallel transport has a positive sign, whereas it enters with a negative sign in the difference of the relative phases $a_J$, since the parallel-transport phase is subtracted in the definition of $a_J$.

\begin{proposition}[Finite variation and infinitesimal variation]\label{prop:variation}
Set
\[
R_J(f)
:=
\int_M f\,r_h^0(J)\,d\nu.
\]
Then, for $f\in\mathfrak g_h$,
\begin{align}
R_{J_s}(f)-R_{J_0}(f)
&=
-\int_M f\beta_s(X_h)\,d\nu
\label{eq:Rfinite}\\
&=
-\frac12
\int_0^s
\int_M
f\ell_{J_\tau}(\dot J_\tau)(X_h)\,d\nu\,d\tau.
\notag
\end{align}
In particular,
\begin{equation}\label{eq:Rvariation}
D_JR_J[A](f)
=
-\frac12
\int_M
f\ell_J(A)(X_h)\,d\nu.
\end{equation}
\end{proposition}

\begin{proof}
Multiply \eqref{eq:phasecomparison} by $f$, integrate over $M$, and divide by $T$.
Since both $f$ and $d\nu$ are invariant under $\phi_t$, for every $T>0$ we have
\[
\frac1T
\int_M
f
\left(
\int_0^T
\beta_s(X_h)(\phi_tx)\,dt
\right)
d\nu
=
\int_M
f\beta_s(X_h)\,d\nu.
\]
By \eqref{eq:polarerror}, the error term is $O(T^{-1})$ even after integration.
Using the weak-$*$ convergence in Proposition \ref{prop:asymptotic}, we therefore obtain \eqref{eq:Rfinite}.
Differentiating its right-hand side with respect to $s$ gives \eqref{eq:Rvariation}.
Notice that no interchange between differentiation with respect to $J$ and the long-time limit is involved.
\end{proof}

\section{Construction of the weighted moment map}\label{sec:momentproof}

\subsection{The differential condition for the moment map}

Set
\[
w_J=\Delta_Jh+|dh|_{g_J}^{\,2}=\diver_\nu(\grad_{g_J}h),\qquad
 \grad_{g_J}h=-JX_h.
\]

\begin{lemma}\label{lem:localvariation}
Let $F_J=S_J-w_J$. Then, for every $A\in T_J\A(M,\omega)$ and
$f\in\gh$, one has
\begin{equation}\label{eq:localvariation}
D_J\int_M fF_J\dn[A]
=
\Omega_{h,J}(A,V_f)
-\int_Mf\ell_J(A)(X_h)\dn.
\end{equation}
\end{lemma}

\begin{proof}
By \eqref{eq:Ricvar},  one has
\[
D_JS_J[A]=\Lambda_\omega d\ell_J(A).
\]
Integration by parts therefore gives
\[
D_J\int_M fS_J\dn[A]
=
-\int_M\ell_J(A)(X_f)\dn
-\int_Mf\ell_J(A)(X_h)\dn.
\]
Using \eqref{eq:Lambdachange} and \eqref{eq:Lambdamoment} together with
$\ell_J(A)=\Lambda_{dV}(J, A)$ and $\nu=e^hdV$, 
we rewrite the right-hand side as
\[
\Omega_{h,J}(A,V_f)
+\int_Mdh(AX_f)\dn
-\int_Mf\ell_J(A)(X_h)\dn.
\]
On the other hand, since $\grad h=-JX_h$, we have
\[
D_Jw_J[A]=-\diver_\nu(AX_h).
\]
Hence
\[
D_J\int_Mfw_J\dn[A]
=
\int_Mdf(AX_h)\dn
=
\int_Mdh(AX_f)\dn.
\]
The last equality follows from the symmetry relation
\[
\omega(u,Av)=\omega(v,Au).
\]
Subtracting the two identities yields \eqref{eq:localvariation}.
\end{proof}

\begin{proof}[Proof of the differential condition in Theorem \ref{thm:weighted}]
We rewrite \eqref{eq:mainmoment} as
\[
\pair{\mu_h(J)}f
=
\int_MfF_J\dn-2R_J(f).
\]
By Lemma \ref{lem:localvariation} and \eqref{eq:Rvariation}, the terms involving
\[
\int_Mf\ell_J(A)(X_h)\dn
\]
cancel, and hence
\[
D_J\pair{\mu_h(J)}f[A]
=
\Omega_{h,J}(A,V_f).
\]
The density in \eqref{eq:mainmoment} lies in $L^\infty(M,d\nu)$, so $\mu_h(J)$ belongs to the continuous dual $\mathfrak g_h'$.
Moreover, applying \eqref{eq:Rfinite} to a smooth local family of paths shows that, for each fixed $f$, the pairing
\[
J\longmapsto \pair{\mu_h(J)}f
\]
depends smoothly on $J$.
\end{proof}

\subsection{Infinitesimal equivariance}
The differential condition alone does not imply equivariance.
Here, equivariance follows from the coboundary representation established in Section \ref{sec:cocycle}.
By the orbit formula \eqref{eq:orbitformula} and the identity
\[
\rho_{h,J}
=
\rho_J+\frac12d(dh\circ J),
\]
we obtain
\begin{equation}\label{eq:weighted-orbit}
\Omega_{h,J}(V_f,V_g)
=
-2C_{\rho_J,h}(f,g)
+\int_Mw_J\{f,g\}\dn.
\end{equation}
Indeed, setting $\alpha=dh\circ J$, and using the fact that $X_f$ and $X_g$ preserve $\nu$, we find
\[
\int_M d\alpha(X_f,X_g)\dn
=
-\int_M\alpha(X_{\{f,g\}})\dn
=
-\int_Mw_J\{f,g\}\dn.
\]
For the last equality, we used
\[
\alpha(X_k)=dh(JX_k)
=
-g_J(\grad h,\grad k)=dk(JX_h),
\]
together with integration by parts.
On the other hand, \eqref{eq:Ricprimitive} gives
\begin{equation}\label{eq:mu-bracket}
\pair{\mu_h(J)}{\{f,g\}}
=
2C_{\rho_J,h}(f,g)
-\int_Mw_J\{f,g\}\dn.
\end{equation}
Combining \eqref{eq:weighted-orbit} and \eqref{eq:mu-bracket}, we obtain
\[
\Omega_{h,J}(V_f,V_g)
+
\pair{\mu_h(J)}{\{f,g\}}
=
0.
\]
This proves infinitesimal equivariance.

\subsection{Equivariance at the group level}

Let $\phi_t\in G_h$ be a path generated by a time-dependent Hamiltonian
$f_t\in\gh$.
Set
\[
J_t=\phi_t^*J,
\qquad
k_t=f_t\circ\phi_t.
\]
Then one has
$
\dot J_t=\Lie_{X_{k_t}}J_t.
$
For a fixed $g\in\gh$, define
$
g_t=g\circ\phi_t.
$
Then
$
\dot g_t=\{k_t,g_t\}.
$
It follows that
\begin{align*}
\frac{d}{dt}\pair{\mu_h(J_t)}{g_t}
&=
\Omega_{h,J_t}(V_{k_t},V_{g_t})
+\pair{\mu_h(J_t)}{\{k_t,g_t\}}\\
&=0.
\end{align*}
Comparing the values at $t=0$ and $t=1$, and choosing
$
g=f\circ\phi_1^{-1},
$
we obtain \eqref{eq:equivgroup}.
This completes the proof of Theorem \ref{thm:weighted}.

We emphasize that no assumption on the simple connectedness of $G_h$, nor any condition on $\pi_1(G_h)$, is required.

\section{Regularity, invariant structures, and nonexistence of adapted almost complex structures}\label{sec:regularity}

\subsection{The density and the range of the moment map}

The construction in this paper directly yields the density
\[
m_h(J)
=
S_J-\Delta_Jh-|dh|_{g_J}^2-2r_h^0(J)
\in L^\infty(M,d\nu),
\]
together with the corresponding element of $\gh'$.
The function $r_h^0(J)$ is uniquely determined almost everywhere as the long-time limit, but the density representing the same functional need not be unique.
For instance, adding an element of
\[
\mathcal N_h
=
\left\{
u\in L^\infty(M,d\nu):
\int_Mfu\dn=0
\quad\text{for all } f\in\gh
\right\}
\]
does not change the moment map.
In particular, the constant functions belong to $\mathcal N_h$.

The present argument does not imply, in general, that $m_h(J)$ is smooth.
At the same time, it does not rule out the existence of a smooth density representing the same functional.
One should distinguish the regularity of the density in the $x$-variable from the smooth dependence, in the $J$-variable, of the functional obtained by pairing it with each $f$.
However, when one restricts to almost complex structures $J$ that are invariant under the flow of $X_h$, the density $m_h(J)$ is a $C^\infty$ function.

\subsection{Restriction to invariant structures}

\begin{proposition}\label{prop:invariant}
Suppose that
\[
\Lie_{X_h}J=0.
\]
Then, with the normalization determined by the canonical connection,
\begin{equation}\label{eq:invariantdensity}
r_h^0(J)=\frac12\Delta_Jh,
\qquad
m_h(J)=S_J-2\Delta_Jh-|dh|_{g_J}^2.
\end{equation}
In particular, if $h$ is constant, then $r_h^0(J)=0$, and one recovers the usual Fujiki--Donaldson moment map.
\end{proposition}

\begin{proof}
Since $\phi_t$ preserves both $\omega$ and $J$, it acts by isometries of $g_J$, and $D\phi_t$ is unitary.
Let $\nabla^g$ denote the Levi--Civita connection of $g_J$.
Then
$\phi_t^*g=g.$
Thus \(X_h\) is a Killing vector field, and
$
  D\phi_t(x):(T_xM,J_x,g_x)
  \longrightarrow(T_{\phi_t x}M,J_{\phi_t x},g_{\phi_t x})
$
is unitary.
Being unitary, however, is different from being parallel transport for
a connection. A map may preserve lengths and angles while still rotating
relative to parallel transport. The function \(a_J(t,x)\) measures the
determinant phase of this relative rotation.
More precisely, consider the canonical Hermitian connection
\begin{equation}\label{eq:hermitian-connection}
  \nablau=\nabla^g-\frac12J(\nabla^g J).
\end{equation}
The connection induces a unitary connection on the determinant line
\[
  L_J=\Lambda_{\C}^n(TM,J).
\]
If \(H_J(t,x)\) denotes its parallel transport along
\(s\mapsto\phi_sx\), then
\begin{equation}\label{eq:phase-definition}
  \Lambda_{\C}^nD\phi_t(x)
  =e^{\,ia_J(t,x)}H_J(t,x),
  \qquad a_J(0,x)=0.
\end{equation}
Here \(a_J(t,x)\) is the continuous real-valued lift of the phase as
\(t\) varies, rather than merely an angle modulo \(2\pi\).
The almost K\"ahler condition implies
\[
\tr((\nabla^g J)v)=0,
\]
\cite[Theorem 2.6(iii)]{GST}.
Fix \(v\in T_xM\) and consider the vector field along
\(\gamma(t)=\phi_tx\) defined by
\[
  V(t)=D\phi_t(x)v.
\]
Since $\dot\gamma=X_h$, 
the variational equation for the flow, together with the torsion-free
property of \(\nabla^g\), gives
$\nabla^g_{\dot\gamma}V=\nabla^g_VX_h.$
It follows that
\begin{align*}
  \widehat\nabla_{\dot{\gamma}}V
  &=\nabla^g_VX_h-\frac12J(\nabla^g_{X_h}J)V\\
  &=K(V), 
\end{align*}
where $$  K(V)=\nabla^g_V X_h-\frac12J(\nabla^g_{X_h}J)V.$$
Thus \(K\) is the rate of change of \(D\phi_t\), measured relative to
parallel transport for \(\nablau\). Since both \(D\phi_t\) and parallel
transport are unitary, this relative infinitesimal generator belongs
to \(\mathfrak u(n)\).
The logarithmic derivative of a determinant is the trace of the
infinitesimal generator. Taking imaginary parts therefore gives
\begin{equation}\label{eq:phase-derivative}
  \frac{\partial}{\partial t}a_J(t,x)
  =\Imag\tr_{\C}K_{\phi_tx}.
\end{equation}
Since $K$ is complex-linear and $\tr K=0$, \eqref{eq:complex-trace} gives 
\begin{equation}\label{eq:complex-real-trace}
  \Imag\tr_{\C}K=-\frac12\tr(JK).
\end{equation}
Since $
  \tr(\nabla^g_{X_h}J)
  =X_h(\tr J)=0, $
the imaginary part of the complex trace of $K$ is
$$
\operatorname{Im}\tr_\C K
=
-\frac12\tr(J\nabla^g X_h)
$$
The divergence is given by
\begin{align}
  \text{\rm div}_{g_J}(JX_h)
  &=\tr\bigl[v\mapsto\nabla^g_v(JX_h)\bigr]\notag\\
  &=\tr\bigl[v\mapsto(\nabla^g_vJ)X_h\bigr]
    +\tr(J\nabla^g X_h)\notag\\
  &=\tr(J\nabla^g X_h).
  \label{eq:divergence}
\end{align}
The final equality follows from the almost K\"ahler identity.  Since $JX_h=-\text{\rm grad}_{g_J} h$ and $\Delta_Jh=\text{\rm div}_{g_J}(\text{\rm grad}_{g_J} h)$, we have
$$
\operatorname{Im}\tr_\C K=-\frac12\text{\rm div}_{g_J}(JX_h)
=
\frac12\Delta_Jh.
$$
Since \(\phi_t\) is the Hamiltonian flow of \(h\),
\[
  h\circ\phi_t=h.
\]
Since \(\phi_t\) is also an isometry, the Laplacian commutes with its
pullback. Therefore
\[
  (\Delta_Jh)\circ\phi_t
  =\Delta_J(h\circ\phi_t)
  =\Delta_Jh.
\]
Integrating \eqref{eq:phase-derivative} with \(a_J(0,x)=0\) yields
\begin{align}
  a_J(T,x)
  &=\frac12\int_0^T\Delta_Jh(\phi_tx)\,dt\notag\\
  &=\frac{T}{2}\Delta_Jh(x).
  \label{eq:exact-phase}
\end{align}
Consequently,
\begin{equation}\label{eq:long-time-density}
  {
  r_h^0(x)
  =\lim_{T\to\infty}\frac{a_J(T,x)}{T}
  =\frac12\Delta_Jh(x).
  }
\end{equation}
\end{proof}

The expression for $m_h(J)$ in \eqref{eq:invariantdensity} is the weighted scalar-curvature density corresponding to the convention
\[
d\nu=e^h\,dV
\]
and to the sign convention for the Laplacian adopted in this paper.
For a general $J$, the term $\Delta_Jh/2$ must be replaced by the long-time phase $r_h^0(J)$.

\subsection{Nonexistence of adapted almost complex structures}

We prove the following proposition.

\begin{proposition}
If $h$ is a Morse function on $(M,\omega)$ with at least one
critical point of odd Morse index, then no $\omega$-compatible
almost complex structure $J$ satisfies $\Lie_{X_h}J=0$.
\end{proposition}

\begin{proof}
Suppose, for a contradiction, that there exists an
$\omega$-compatible almost complex structure $J$ satisfying
$\Lie_{X_h}J=0$.
Let $x$ be a critical point of $h$ of odd Morse index.

Choose a connection $\nabla$ preserving both $\omega$ and $J$.
The Hessian of $h$ at $x$ is
$
 H_h=(\nabla dh)_x.
$
Using the convention $\iota_{X_h}\omega=-dh$, we obtain
$
 H_h(u,v)=-\omega_x(\nabla_uX_h,v),
$ for $ u,v\in T_xM.
$
Since $X_h(x)=0$, the condition $\Lie_{X_h}J=0$ gives
\[
 0=(\Lie_{X_h}J)_x u
   =J\nabla_uX_h-\nabla_{Ju}X_h.
\]
Consequently,
\[
 \begin{aligned}
 H_h(Ju,Jv)
 &=-\omega_x(\nabla_{Ju}X_h,Jv)\\
 &=-\omega_x(J\nabla_uX_h,Jv)\\
 &=-\omega_x(\nabla_uX_h,v)
 =H_h(u,v).
 \end{aligned}
\]
Thus $H_h$ is $J$-invariant.

Let $S:T_xM\to T_xM$ be the self-adjoint endomorphism determined by
\[
 H_h(u,v)=g_x(Su,v),
 \qquad g=\omega(\cdot,J\cdot).
\]
The $J$-invariance of $H_h$ implies that $SJ=JS$.
Hence the direct sum $N_-(H_h,x)$ of the negative eigenspaces
of $S$ is $J$-invariant and therefore has even real dimension.
It follows that
\[
 \operatorname{ind}_x(h)=\dim_{\R}N_-(H_h,x)
\]
is even, contradicting the choice of $x$.
\end{proof}

\section{The exact volume-preserving diffeomorphism group and the first Chern class}
\label{sec:exactvolume}

In this section, rather than fixing $\omega$ and $h$, we consider a closed,
connected, oriented manifold $M^m$, with $m=2n\ge4$, a positive volume form
$\nu$, and a nonempty connected component $\J_*$ of $\J(M)$.
The Lie algebra $\gex$ is defined by \eqref{eq:gex}, with the usual Lie
bracket of vector fields.
We adopt the following convention for the moment map:
\begin{equation}\label{eq:ext-convention}
  D_J\pair{\mu(J)}v[A]=\Omega_{\nu,J}(A,\Lie_vJ).
\end{equation}

\subsection{The differential condition always admits a solution}

Fix $J_0\in\J_*$ and write $\rho_J^\nu=\Ric_{\nu,J}$.
Since $c_1(TM,J)$ is constant on $\J_*$, the form
$\rho_J^\nu-\rho_{J_0}^\nu$ is exact.
For $v\in\gex$, choose an $(m-2)$-form $\alpha_v$ satisfying
$d\alpha_v=\iota_v\nu$, and define
\begin{equation}\label{eq:weakext}
  \pair{\mu_0(J)}v
  =-2\int_M(\rho_J^\nu-\rho_{J_0}^\nu)\wedge\alpha_v.
\end{equation}
Since any two choices of $\alpha_v$ differ by a closed form, this
expression is independent of the choice.
To verify continuity, choose an auxiliary Riemannian metric and use the
continuous linear right inverse $d^*G$ provided by Hodge theory to take
$\alpha_v=d^*G(\iota_v\nu)$.

By \eqref{eq:Ricvar} and \eqref{eq:Lambdamoment},
\[
  D_J\pair{\mu_0(J)}v[A]
  =-\int_Md\Lambda_\nu(J,A)\wedge\alpha_v
  =-\int_M\Lambda_\nu(J,A)(v)\,\nu
  =\Omega_{\nu,J}(A,\Lie_vJ).
\]
Thus, the differential condition alone always admits a solution,
regardless of $c_1$.
However, there is a clear obstruction to the existence of an equivariant moment map.

\subsection{The obstruction to equivariance}
For a closed $2$-form $\eta$, define the Lichnerowicz cocycle by
\[
  C_{\eta,\nu}(u,v)=\int_M\eta(u,v)\,\nu,
  \qquad u,v\in\gex.
\]
Since $u$ and $v$ preserve the volume form, we have
\[
  d(\iota_u\iota_v\nu)=\iota_{[u,v]}\nu,
  \qquad
  \eta\wedge\iota_u\iota_v\nu=-\eta(u,v)\nu.
\]
Consequently, the map $\mu_0$ defined by \eqref{eq:weakext} satisfies
\[
  \pair{\mu_0(J)}{[u,v]}
  =2C_{\rho_J^\nu,\nu}(u,v)-2C_{\rho_{J_0}^\nu,\nu}(u,v).
\]
Combining this with the orbit formula \eqref{eq:orbitformula} gives
\begin{equation}\label{eq:ext-defect}
  \Omega_{\nu,J}(\Lie_uJ,\Lie_vJ)+\pair{\mu_0(J)}{[u,v]}
  =-2C_{\rho_{J_0}^\nu,\nu}(u,v).
\end{equation}

On the connected component $\J_*$, any two maps satisfying the
differential condition differ by a continuous linear functional
$b\in\gex'$ that is independent of $J$.
Hence, an equivariant moment map exists if and only if
\begin{equation}\label{eq:ext-obstruction}
  [C_{\rho_{J_0}^\nu,\nu}]=0
  \quad\text{in }\Hcont(\gex,\R).
\end{equation}
Indeed, the equivariance defect of $\mu_0+b$ is
$-2C_{\rho_{J_0}^\nu,\nu}(u,v)+b([u,v])$, and its vanishing is precisely
the condition that the cocycle be a continuous coboundary.

\subsection{A criterion via continuous cohomology}

We use the following known result.
\begin{theorem}[Janssens--Ryvkin--Vizman]\label{thm:JRV}
Let $M$ be a closed, oriented $m$-dimensional manifold with $m\ge3$.
Then the map
\begin{equation}\label{eq:JRV}
  \HdR^2(M;\R)\longrightarrow\Hcont(\mathfrak g_{\mathrm{ex},\nu},\R),
  \qquad [\eta]\longmapsto[C_{\eta,\nu}]
\end{equation}
is an isomorphism.
\end{theorem}
This formulation follows by combining the classification of universal
central extensions in \cite[Theorem 5.2, Corollary 5.3]{JRV} with the
identification
$(\HdR^{m-2}(M;\R))'\simeq\HdR^2(M;\R)$
provided by Poincar\'e duality for closed manifolds.
It is essential here that the cohomology is continuous, rather than
algebraic.

\begin{proof}[Proof of Theorem \ref{thm:exactvolume}]
Suppose that an equivariant moment map exists.
Then \eqref{eq:ext-obstruction} holds.
By the injectivity in Theorem \ref{thm:JRV},
\[
  [\rho_{J_0}^\nu]=0\quad\text{in }\HdR^2(M;\R).
\]
By \eqref{eq:ricciclass}, this means that the real first Chern class
vanishes.

Conversely, suppose that $c_1(TM,J)=0$ in $H^2(M,\R)$.
Then $\rho_J^\nu$ is exact, and
\[
  \pair{\mu_{\mathrm{ex}}(J)}v
  =-2\int_M\rho_J^\nu\wedge\alpha_v
\]
is independent of the choice of $\alpha_v$ and defines an element of
the continuous dual.
The variation calculation above shows that it satisfies the differential
condition.
Moreover, for any diffeomorphism $\phi$ preserving both $\nu$ and the
component $\J_*$, the identity
$\rho_{\phi^*J}^\nu=\phi^*\rho_J^\nu$ gives
\[
  \pair{\mu_{\mathrm{ex}}(\phi^*J)}v
  =-2\int_M\rho_J^\nu\wedge(\phi^{-1})^*\alpha_v
  =\pair{\mu_{\mathrm{ex}}(J)}{\phi_*v}.
\]
Here we have used
$d((\phi^{-1})^*\alpha_v)=\iota_{\phi_*v}\nu$.
In particular, the moment map is equivariant under
$G_{\mathrm{ex},\nu}$.
\end{proof}

\begin{remark}[Relation to existing Ricci moment maps]
The variation formula for the Ricci form and the moment map for the
exact volume-preserving diffeomorphism group in the case of vanishing
real first Chern class were established in \cite{GST}.
For a general framework treating diffeomorphism group actions in terms
of moment maps on spaces of sections, see \cite{DonaldsonMM}.
The introduction to \cite{GST} also points out the relation between the
first Chern class and the dual space.
In the present paper, we use the classification in \cite{JRV} to
establish necessity for arbitrary equivariant moment maps with values
in the continuous dual, rather than restricting attention to those
represented by smooth $2$-forms.
\end{remark}

\begin{remark}[Real dimension two]
Theorem \ref{thm:JRV} does not apply when $m=2$, and the necessity
statement in Theorem \ref{thm:exactvolume} also fails in this dimension.
Indeed, on a surface equipped with an area form $\nu$, the exact
volume-preserving diffeomorphism group is the Hamiltonian
diffeomorphism group, and every almost complex structure compatible
with the orientation is compatible with $\nu$.
The usual scalar curvature moment map exists, for example, on $S^2$,
even though $c_1(TS^2)\ne0$.
\end{remark}


\end{document}